\documentclass[12pt,reqno]{amsart}
\usepackage{amsmath, amssymb, latexsym, amscd, amsthm,amsfonts,amstext}
\usepackage[mathscr]{eucal}

\newtheorem{theorem}{Theorem}[section]

\newtheorem{lemma}[theorem]{Lemma}

\newtheorem{definition}[theorem]{Definition}
\newtheorem{remark}[theorem]{Remark}

\def\underset#1#2{{\mathrel{\mathop {{}_{} {#2}}\limits_{{#1}_{}}}}}
\def\upplim_#1{\underset{#1}{\overline\lim}\;}
\def\lowlim_#1{\underset{#1}{\underline\lim}\;}
\newcommand{\rank}{\mathrm{rank}}

\numberwithin{equation}{section}
\begin{document} 
\title[Schmidt's subspace theorem for moving hypersurface targets]{The degenerate Schmidt's subspace theorem for moving hypersurface targets.} 

\author{Giang Le}

\setlength{\baselineskip}{16pt}
\maketitle

\begin{center}
{\it Department of Mathematics, Hanoi National University of Education,}

{\it 136-Xuan Thuy, Cau Giay, Hanoi, Vietnam.}

\textit{E-mail: legiang01@yahoo.com}
\end{center}

\begin{abstract} {Our goal is to give Schmidt's subspace theorem for moving hypersurface targets in subgeneral position in projective varieties.}
\end{abstract}

\def\thefootnote{\empty}

\setlength{\baselineskip}{16pt}
\maketitle
\section{Introduction}
Schmidt's subspace theorem is a very powerful tool from Diophantine approximation which has many significant applications to Diophantine equations. Its original form can be referred to as ''Schmidt's subspace theorem with fixed targets'' since the finitely many targets can be thought of as remaining fixed as an approximating points moves through infinitely many points. One direction to generalize the subspace theorem is to allow ''targets'' to vary slowly.

In 1980s, due to the work of Vojta, Osgood, Lang, etc, people have started to realize that there is a striking analogue between Nevanlinna theory and Diophantine approximation. Vojta has compiled a dictionary about this connection. Via this dictionary, Cartan's Second Main Theorem corresponds to Schmidt's subspace theorem. A growing understanding of these analogue has motivated the development in both subjects. 

C. Osgood (see \cite{O1,O2}) and N. Steinmetz (see \cite{Ste}) proved "Second Main theorem with moving targets". This was Vojta's motivation for Roth's theorem with moving targets (see \cite{V2}). Later, M. Ru and Vojta (see \cite{RV}) extended this theorem to a version of Schmidt's subspace theorem with moving targets which corresponds to Ru-Stoll's result \cite{RS} in Nevanlinna theory. 

To state Schmidt's subspace theorem, we first  introduce some standard notations in Diophantine geometry. For details concerning the Diophantine Geometry, we refer the reader to \cite{1}, \cite{V1}.  Through this paper, let $k$ be a number field. Denote by $M_k$ the set of places (equivalent classes of absolute values) of $k$ and write $M_k^{\infty}$ for the set of archimedean places of $k$. For $v\in M_k$, we choose the normalized absolute value $|.|_v$ such that $|.|_v=|.|$ on $\mathbb{Q}$ (the standard absolute value) if $v$ is archimedean, whereas for $v$ non-archimedean $|p|_v=p^{-1}$ if $v$ lies above the rational prime $p$. For a valuation $v$ of $k$, denote by $k_v$ the completion of $k$ with respect to $v$ and set $n_v:=[k_v:\mathbb{Q}_v]/[k:\mathbb{Q}]$. We put $\|x\|_v=|x|_v^{n_v}.$ These absolutes values satisfy the product formula
$$\prod_{v\in M_k}\|x\|_v=1 \,\text{for}\, x\in k^*.$$

For  $x=[x_{0}:\ldots:x_{M}]\in \mathbb{P}^M(k)$, we put
$$\|x\|_v:=\max(\|x_0\|_v,\ldots,\|x_M\|_v),\quad v\in M_k.$$
 Then the absolute logarithmic height of $x$ is defined by
$$h(x)=\sum_{v\in M_k}\log\|x\|_v.$$ 
By the product formula, this does not depend on the choice of homogeneous coordinates $[x_0:\ldots:x_M].$ If $x\in k^*,$ we define the absolute logarithmic height of $x$ by
$$h(x)=\sum_{v\in M_k}\log^+\|x\|_v.$$ 

We also set the convenient notation
$$\epsilon_v(r)=\begin{cases}&r\,\mbox{ if}\; v\mbox{ is archimedean}\\
&1\,\mbox{ if}\; v\;\mbox{ is non-archimedean}.\end{cases}$$
With this notation, the triangle inequality can be written uniformly as follow
$$\|a_1+\cdots+a_r\|_v\leq\epsilon^{n_v}_v(r)\max\{\|a_1\|_v,\ldots,\|a_r\|_v\}, \forall a_i\in k, i=1,\ldots, r.$$

For a positive integer $d$, we set
$$\mathcal{T}_d:=\{(i_0,\ldots,i_M)\in\mathbb{N}^{M+1}_0: i_0+\cdots+i_M=d\}.$$

Let $Q=\sum_{I\in\mathcal{T}_{d}}a_{I}x^I$ be a homogeneous polynomial of degree d in $k[X_0,\ldots,X_M],$ where $x^I=x_0^{i_0}\cdots x_M^{i_M}$ for $x=(x_0,\ldots,x_M)$ and $I=(i_0,\ldots,i_M).$ Denote by $\|Q\|_v=\max\{\|a_I\|_v\}.$ The height of $Q$ is defined by
$$h(Q)=\sum_{v\in M_k}\log\|Q\|_v.$$
Then, for  every $v\in M_k$, the Weil function $\lambda_{Q,v}$ is defined by
$$\lambda_{Q,v}(x)=\log\dfrac{\|x\|_v^d.\|Q\|_v}{\|Q(x)\|_v},\quad x\in\mathbb{P}^M(k)\backslash\{Q=0\}.$$

We now state a variant statement of Ru-Vojta's result \cite{RV}which is more convenient to use. 

Let $\Lambda$ be an infinite index set. A collection of points $\{x(\alpha)\in\mathbb{P}^M(k)|\alpha\in\Lambda\}$ will be regarded as a map $x: \Lambda\longrightarrow\mathbb{P}^M(k)$.
\vskip0.2cm
\noindent
\textbf{Theorem A (Schmidt's subspace theorem for moving hyperplane targets).}  {\it Let k be a number field, $M_k^\infty\subset S$  be a finite set of places of k , let $\epsilon>0.$ Let $\Lambda$ be an infinite index set.  Let $L_1,\ldots, L_q$ be moving hyperplanes $\Lambda\longrightarrow (\mathbb{P}^n)^*(k)$ and let $x: \Lambda\longrightarrow\mathbb{P}^n(k)$ be a collection of points such that:


(1) x is non-degenerate over $\mathcal{R}$ with respect to $L_1,\ldots, L_q.$

(2) $h(L_j(\alpha))=o(h(x(\alpha))), j=1,\ldots, q.$

Then, there exists an infinite index subset $A\subset \Lambda$ such that
$$\sum_{v\in S}\max_K\sum_{j\in K}\lambda_{L_j(\alpha),v}(x(\alpha))\leq (n+1+\epsilon)h(x(\alpha)),$$
for all $\alpha\in A,$ where the maximum is taken over all subsets $K$ of $\{1,\ldots, q\}$ such that $L_j(\alpha), j\in K$ are linearly independent over $k$ for each $\alpha\in \Lambda.$
}\\

Ru-Vojta \cite{RV} also studied the more general case in which hyperplanes are located in $m$-subgeneral position.\\
\noindent
 \textbf{Theorem B.}  {\it Let k be a number field, $M_k^\infty \subset S$ be a finite set of places of k, let $\epsilon>0.$ Let $\Lambda$ be an infinite index set. Let $L_1,\ldots, L_q$ be moving hyperplanes $\Lambda\longrightarrow (\mathbb{P}^n)^*(k)$ and let $x: \Lambda\longrightarrow\mathbb{P}^n(k)$ be a collection of points such that:

(1) For  every  $\alpha\in \Lambda$, $L_1(\alpha),\ldots, L_q(\alpha)$ are in m- subgeneral position, that means any $m+1$ linear forms in $\{L_1(\alpha),\ldots, L_q(\alpha)\}$ have no common solutions in $\mathbb{P}^n(k)$.

(2) x is non-degenerate over $\mathcal{R}$ with respect to $L_1,\ldots, L_q.$

(3) $h(L_j(\alpha))=o(h(x(\alpha))), j=1,\ldots, q.$

Then, there exists an infinite index subset $A\subset \Lambda$ such that
$$\sum_{v\in S}\sum_{j=1}^q\lambda_{L_j(\alpha),v}(x(\alpha))\leq (2m-n+1+\epsilon)h(x(\alpha)),$$
for all $\alpha\in A.$
}

The generalizations of the Subspace theorem to projective variety $V$ and hypersurfaces located in general position have been given by  Corvaja-Zannier \cite{CZ} and Evertse-Ferretti \cite{EF08}. The case when $V$ is  projective space is due to Corvaja-Zannier. Later, Ru proved the analytic counter-part of such the results in \cite{Ru04, Ru09}. After that, Dethloff-Tan \cite{DT} and Cherry-Dethloff-Tan \cite{CDT} generalize Ru's results to moving hypersurface targets.  In arithmetic case, G. Le \cite{G2}, Chen-Ru-Yan \cite{CRY3} and Son-Tan-Thin \cite{STT} extended  Corvaja-Zannier-Evertse-Ferretti's result to moving hypersurface targets and projective variety $V$. The case when $V$ is  projective space is due to G. Le \cite{G2} and Chen-Ru-Yan \cite{CRY3}. 

Let $V$ be an irreducible projective subvariety of $\mathbb{P}^M$ defined over $k$ of dimension $n, (n\leq M)$. Let $m\geq n$ be a positive integer. Recall that hypersurfaces $\{D_i\}_{i=1}^q, q>m,$  in $\mathbb{P}^M$  are said to be located in \emph{m-subgeneral position} with respect to $V$ if for any $1\le i_0<\cdots <i_{m}\leq q$,
$$ V(\bar{k})\cap (\bigcap_{j=0}^{m}Q_{i_j}=0)=\emptyset .$$
If $m=n$, they are said to be located in \emph{general position}.

 Recently, Chen, Ru and Yan \cite{CRY2} and Levin \cite{Levin}(Theorem 5.1) generalized Corvaja-Zannier-Evertse-Ferretti' results to projective variety  $V$ and family of hypersurfaces located in $m-$subgeneral position with respect to $V$. This is our motivation to study Schmidt's subspace theorem for moving hypersurface targets in subgeneral position in projective variety.  

To state our results, we first introduce some notations.

Let $A\subset\Lambda$ be an infinite index subset and $a$ be a set-theoretic map $A\longrightarrow k$ . For precisely, we can denote this map by $(A,a)$.
\begin{definition}
Let $A\subset\Lambda$ be an infinite index subset and $C_1,C_2\subset A$ be subsets of $A$ with finite complement.Two pairs $(C_1, a_1)$ and $(C_2, a_2)$ are called equivalent if there is a subset $C\subset C_1\cap C_2$ such that $C$ has finite complement in $A$ and such that the restrictions of $a_1,a_2$ to $C$ coincide. Let $\mathcal {R}^0_A$ be the set of equivalent classes of pairs $(C,a)$. Then $\mathcal{R}^0_A$ has an obvious ring structure. Moreover, we can embed $k$ into $\mathcal{R}^0_A$ as constant functions.
\end{definition}

 A moving hypersurfaces of degree $d$ in $\mathbb{P}^M(k)$ will be regarded as a map $Q: \Lambda\longrightarrow \mathbb{P}(H^0(\mathbb{P}^M(k),\mathcal{O}(d))).$ For every $\alpha\in \Lambda,$ choose $\{a_I(\alpha)\in k\}_{ I\in\mathcal{T}_d}$ such that $Q(\alpha)$ is the hypersurface determined by the equation $\sum_{I\in\mathcal{T}_d}a_I(\alpha)x^I=0.$ If there is no confusion, we use the same notation $Q$ to denote the homogeneous polynomial in  $\mathcal{R}^0_{\Lambda}[X_0,\ldots, X_M]$ defined by  
$$Q(\alpha):=\sum_{I\in\mathcal{T}_d}a_I(\alpha)x^I, \;\text{for all}\; \alpha\in\Lambda.$$

Given moving hypersurfaces $Q_1(\alpha),\ldots, Q_q(\alpha), (\alpha\in \Lambda)$ in $\mathbb{P}^M(k)$ respectively of degrees $d_1,\ldots, d_q,$ choose $a_{j,I}(\alpha)\in k, j=1,\ldots, q, I\in\mathcal{T}_{d_j}$ such that 
$Q_j(\alpha)$ is given by $\sum_{I\in\mathcal{T}_{d_j}}a_{j,I}(\alpha)x^I=0.$ Set
$$ m_j:=|\mathcal{T}_{d_j}|-1, j=1,\ldots, q.$$
Numbering the elements of the set $\mathcal{T}_{d_j}$ from 0 to $m_j, 1\leq j\leq q.$

\begin{definition}{\label{1.1}} An infinite subset $A\subset\Lambda$ is said to be coherent with respect to $\{Q_j\}_{j=1}^q$ if for every polynomial $P\in k[X_{1,0},\ldots,X_{1,m_1},X_{2,0}, \ldots, X_{q,m_q}]$ which is homogeneous in $X_{j,0},\ldots,X_{j,m_j}$ for each $j=1,\ldots,q,$ either $P(a_{1,0}(\alpha),\ldots,a_{q,m_q}(\alpha))$ vanishes for all $\alpha\in A$, or it vanishes for only finitely many $\alpha\in A.$
\end{definition}
\begin{remark} The above definition is independent of the choice of coefficients $a_{j,I}(\alpha)\in k, j=1,\ldots, q, I\in\mathcal{T}_{d_j}$.
\end{remark}
\begin{lemma} There exists an infinite subset $A\in\Lambda$ which is coherent with respect to $\{Q_j\}_{j=1}^q$.
\end{lemma}
The proof of this lemma is similar to the proof of Lemma 1.1 in \cite{RV} without any modifications.

Let $A\subset\Lambda$ be an infinite index subset which is coherent with respect to $\{Q_j\}_{j=1}^q$. If $j\in\{1,\ldots,q\}$ and $\mu,\nu\in\{0,\ldots,m_j\}$ are such that $a_{j,\nu}(\alpha)\not=0$ for at least one $\alpha\in A,$ then the set $\{\alpha\in A|a_{j,\nu}(\alpha)\not=0\}$ has finite complement in $A$ by coherence. Hence, the pair 
$$\{\alpha\in A|a_{j,\nu}(\alpha)\not=0\}\longrightarrow k, \alpha\longrightarrow a_{j,\mu}(\alpha)/a_{j,\nu}(\alpha)$$ lies in $\mathcal{R}_A^0.$
 Moreover, the subring of $\mathcal{R}^0_A$ generated over $k$ by all such pairs is entire. In particularly, for all $a$ belongs to this subring $a(\alpha)=0$ for all $\alpha\in A$ or for only finitely $\alpha\in A$.

\begin{definition}
We define $\mathcal{R}_{A,\{Q_j\}_{j=1}^q}$ to be the quotient field of the above-mentioned entire subring. 
\end{definition}

Note that the field $\mathcal{R}_{A,\{Q_j\}_{j=1}^q}$ is independent of the choice of coefficients.
\begin{remark}{\label{re:1.3}} Let $B\subset A\subset\Lambda$ be two infinite index subsets. Then it's clear that if $A$ is coherent then so is $B$, and $\mathcal{R}_{A,\{Q_j\}^q_{j=1}}=\mathcal{R}_{B,\{Q_j\}^q_{j=1}}.$ 
\end{remark}

Let  $V$ be an irreducible projective subvariety of  $\mathbb{P}^M$ defined over $k$ of dimension $n$. Let $I_V$ be the homogeneous prime ideal of  $k[X_0,\ldots, X_M]$ consisting of all polynomials vanishing identically on  $V$. Let $A\subset \Lambda$ be coherent with respect to $\{Q_j\}_{j=1}^q$. Denote by  $I_{A, V, \{Q_j\}_{j=1}^q}$ the ideal of the ring $\mathcal{R}_{A,\{Q_j\}^q_{j=1}}[X_0,\ldots, X_M]$ generated by $I_V$. 

\begin{definition}
Let $x:\Lambda\longrightarrow V(k)$ be a collection of points. We say that $x$ is  algebraically non-degenerate over $V$ and $ \mathcal{R}_{\{Q_j\}_{j=1}^q}$ if for all infinite subsets $A\subset\Lambda$ that are coherent with respect to $\{Q_j\}_{j=1}^q$, there is no homogeneous polynomial $Q\in \mathcal{R}_{A,\{Q_j\}_{j=1}^q}[X_0,\ldots,X_M]\backslash I_{A, V, \{Q_j\}_{j=1}^q}$ such that $Q(x_o(\alpha),\ldots,x_M(\alpha))=0,$ for all $\alpha\in A$ outside a finite subset of $A$.
\end{definition}

\begin{definition}
Let  $m\geq n$ be a positive integer. We say that a set $\{Q_j\}^q_{j=1}, (q\geq m+1)$ of homogeneous polynomials in $\mathcal{R}^0_{\Lambda}[X_0,\ldots, X_M]$ is in m-subgeneral with respect to $V$ if there exists an infinite subset $A\subset\Lambda$ with finite complement such that for any $1\leq j_0<\cdots< j_m\leq q,$ and $\alpha\in A,$  the system of equations
\begin{align*}
Q_{j_i}(\alpha)(x_0,\ldots,x_M)=0, 0\leq i\leq m
\end{align*}
has no solutions in $V(\bar{k})$, in which $\bar{k}$ is an algebraic closure of k. 
\end{definition}
Our main result is stated as following
\vskip0.2cm
\noindent
\textbf{Main Theorem.}\label{Main}
{\it Let k be a number field, $M_k^\infty \subset S$ be a finite set of places of k, let $q,m,n$ be positive integers with $q>m\geq n$ and $\epsilon>0.$ Let $\Lambda$ be an infinite index set, let $Q_1,\ldots, Q_q$ be moving hypersurfaces in $\mathbb{P}^M(k)$ respectively of degrees $d_1, \ldots, d_q$.  Let  $V$ be an irreducible projective subvariety of  $\mathbb{P}^M$ defined over $k$ of dimension $n$ and let $x: \Lambda\longrightarrow V(k)$ be a collection of points such that:

 (1) The family of  polynomials $Q_1,\ldots, Q_q$ is in m-subgeneral position with respect to $V$;
 
 (2)  x is algebraically non-degenerate over $V$ and $ \mathcal{R}_{\{Q_j\}_{j=1}^q}$;
 
 (3) $ h(Q_j(\alpha))=o(h(x(\alpha)))$ for all  $j=1,\ldots, q.$
 
 Then there exists an infinite index subset $A\subset\Lambda$ such that
$$\sum_{v\in S}\sum_{j=1}^q\dfrac{1}{d_j}\lambda_{Q_j(\alpha),v}(x(\alpha))\leq (m(n+1)+\epsilon)h(x(\alpha)),$$
for all $\alpha\in A.$}

Notice that when $m=n$, i.e, the family of polynomials is in general position with respect to  $V$, our result is weaker than Son-Tan-Thin's result \cite{STT}.
\section{Some Lemmas}
In this section, we will state some lemmas needed for the proof of the Main Theorem.

 Masser and Wustholz  \cite{MW} proved a simple lemma on the solutions of a system of linear equations over  algeraic function fields. We restate it in the number field setting.

Let $k$ be a number field. For  positive integers  $p$ and $q$, we consider the system
\begin{align} \label{e:201}
a_{j1}x_1+\cdots+a_{jp}x_p=0\;\;\; (1\leq j \leq q), 
\end{align}
where $a_{ij}\in k$ not all zero $(1\leq i\leq p, 1\leq j \leq q).$
\begin{lemma}\label{l:21} For an integer t with  $1\leq t \leq p,$ suppose that the system of \eqref{e:201} has a solution $x_1,\ldots, x_p$ in $k$ such that $x_t\not=0$. Then, the system \eqref{e:201} has a solution $x_1,\ldots, x_p$ in $k$ with $x_t\not=0$ and $x_i$ is either a certain signed minor of the matrix $(a_{ij})_{i,j}$ or $0$. Furthermore, for all $1\leq i\leq p$, $v\in M_k,$
$$ \|x_i\|_v \leq \epsilon_v^{n_v}(p^2) \max_{i, j}(\|a_{ij}\|_v,1)^p .$$
\end{lemma}
\begin{proof} Let  $l$ be the rank of the system \eqref{e:201}.  If $l = p$ then
the system has a unique solution $x_1 =\cdots= x_p = 0$ which contradicts the existence of a solution
with $x_t\not =0$. Therefore $1\leq l \leq p-1$ and that the system is equivalent to (possibly after permuting the unknowns)
\begin{align}\label{e:202}
\delta x_i=\delta_{i, l+1} x_{l+1}+\cdots+\delta_{jp}x_p\;\;\; (1\leq i \leq l)
\end{align}
where $\delta (\not=0)$ and $\delta_{ij} (1\leq i \leq l, l+1\leq j\leq p)$ are certain signed minors of order $l$. We have, for all $v\in M_k,$
$$ \|\delta\|_v,\|\delta_{ij}\|_v \leq \epsilon_v^{n_v}(l^2) \max_{i, j}\|a_{ij}\|_v^l \leq \epsilon_v^{n_v}(p^2) \max_{i, j}(\|a_{ij}\|_v,1)^p.$$
We distinguish two cases. Firstly, suppose $1\leq t \leq l $. Then $\delta_{ts}\not=0$ for some $l+1\leq s \leq p$; otherwise, \eqref{e:202} implies that $x_t=0$ for all solutions of \eqref{e:201}. Then, the required solution of \eqref{e:201} can be taken to
$$ x_s=\delta, x_j=0 (l+1\leq j \leq p, j\not=s), x_i=\delta_{is} (1\leq i \leq l). $$
Secondly, suppose $l+1\leq t \leq p.$ Then the required solution is
$$ x_t=\delta, x_j=0 (l+1\leq j \leq p, j\not=t), x_i=\delta_{it} (1\leq i \leq l). $$
This proves the lemma.
\end{proof}
Let $\{Q_j\}_{j=1}^q$ be  a set of homogeneous polynomials in $\mathcal{R}^0_\Lambda[X_0,\ldots, X_M].$ Let $A\subset \Lambda$  be coherent with respect to $\{Q_j\}_{j=1}^q$.
\begin{lemma}\label{l:22}
 We consider the system
\begin{align*}
a_{j1}x_1+\cdots+a_{jp}x_p=0\;\;\; (1\leq j \leq q), 
\end{align*}
where  $a_{ji}\in \mathcal{R}_{A,\{Q_j\}_{j=1}^q}$. If the above system has a solution $x_1,\ldots, x_p$ in $ \mathcal{R}^0_{\Lambda}$ such that $x_t\not=0$. Then, the system has a solution $x_1,\ldots, x_p$ in $\mathcal{R}_{A,\{Q_j\}_{j=1}^q}$ with $x_t\not=0.$ Furthermore, for all $v\in M_k$, $1\leq i\leq p, \alpha \in A$
$$ \|x_i(\alpha)\|_v\leq \epsilon_v^{n_v}(p^2) \max_{i,j}(\|a_{ij}(\alpha)\|_v,1)^p. $$
\end{lemma}
\begin{proof}
Without loss of generality, we can assume that $x_t(\alpha)\not=0$ for all $\alpha \in A.$ 
Set $l=\max_\alpha\rank (a_{ij}(\alpha))$. Then $l<p.$ If $l=\rank (a_{ij}(\alpha))$ then there exists a non-zero minor of order $l$ of matrix $(a_{ij}(\alpha))$. By coherence of $A$, such the minor is different from 0 for all but finitely many $\alpha \in A$. Using the similar arguments as in proof of Lemma \ref{l:21}, we have the required solution.
\end{proof}
Let  $V$ be an irreducible projective subvariety of  $\mathbb{P}^M$ defined over $k$ of dimension $n$ and degree $\triangle$. Let $I_V$ be the homogeneous prime ideal of  $k[X_0,\ldots, X_M]$ consisting of all polynomials vanishing identically on  $V$. Let $P_1,\ldots, P_r$ be generators of $I_V.$ Denote by  $I_{A, V, \{Q_j\}_{j=1}^q}$ the ideal of the ring $\mathcal{R}_{A,\{Q_j\}^q_{j=1}}[X_0,\ldots, X_M]$ generated by $I_V$. 

\begin{lemma}\label{l:23}
Let $P\in \mathcal{R}_{A, \{Q_j\}_{j=1}^q}[X_0,\ldots, X_M]$ be a homogeneous polynomial of degree $d$. Then $P(\alpha)\in I_V$ for all but finitely $\alpha \in A$ if and only if $P\in I_{A, V, \{Q_j\}_{j=1}^q}.$
\end{lemma}
\begin{proof}
The part "if" is obvious. We prove the part "only if".

By passing to infinite subset of $A$, we can  assume that  $P(\alpha)\in I_V$ for all $\alpha \in A$. Since $P(\alpha)\in I_V$, there exist $A_1(\alpha),\ldots, A_r(\alpha)\in  k[X_0,\ldots, X_M]$ (which we may assume to be homogeneous) such that $$ P(\alpha)=A_1(\alpha)P_1+\ldots+A_r(\alpha)P_r.$$
We rewrite the above equation in the following form
\begin{align}\label{e:204}
a(\alpha) P(\alpha)=A_1(\alpha)P_1+\ldots+A_r(\alpha)P_r, 
\end{align}
where $a: A\longrightarrow k, a(\alpha)=1$ for all $\alpha \in A.$ Without loss of generality, we can assume that $A_i(\alpha)$ is homogeneous of  degree $d-\deg P_i$, $1\leq i \leq r$. We may regard \eqref{e:204} as a system of linear equations in coefficients of $A_i(\alpha), i=1,\ldots, r$ and $a(\alpha)$ with $a(\alpha)\not=0$. Thus, we can apply lemma \ref{l:22} to this system. Thus, we can choose $A_i\in \mathcal{R}_{A, \{Q_j\}_{j=1}^q}[X_0,\ldots, X_M], 0\not =a(\alpha)\in \mathcal{R}_{A, \{Q_j\}_{j=1}^q}$ satisfying \eqref{e:204}. Hence, $P\in I_{A, V, \{Q_j\}_{j=1}^q}.$
\end{proof}
\begin{lemma}\label{l:24}
   Let  $\{Q_j\}_{j=0}^m$ be a set of homogeneous polynomials  in $\mathcal{R}^0_{\Lambda}[X_0,\ldots, X_M]$ located in $m-$subgeneral position with respect to $V$. Let $A$ be coherent with respect to $\{Q_j\}$. Assume that all coefficients of $Q_j, 0\leq j\leq m$ belong to the field $\mathcal{R}_{A,\{Q_j\}_{j=0}^m}$. Then, for each $0\leq i\leq M$, there exist $r_i\in \mathbb{N}$, homogeneous polynomials $A_{il}$ of degree $r_i-\deg P_l$, $B_{ij}$ of degree $r_i-\deg Q_j$ in $ \mathcal{R}_{A,\{Q_j\}_{j=0}^m}[X_0,\ldots, X_M] (1\leq l \leq r, 0\leq j \leq m)$ such that
$$X_i^{r_i}=\sum_{l=1}^rA_{il}P_l+\sum_{j=0}^mB_{ij}Q_j, i=0,\ldots, M.$$
\end{lemma}
\begin{proof}
Since $\{Q_j\}_{j=0}^m$ is in $m-$subgeneral position with respect to $V$, we have $P_1,\ldots, P_r; $\\
$ Q_0(\alpha),\ldots, Q_m(\alpha)$ have no common zeros in $\mathbb{P}^M(\bar{k})$ for all but finitely $\alpha\in A.$ By passing to infinite subset of $A$, we can assume that $P_1,\ldots, P_r; Q_0(\alpha),\ldots, Q_m(\alpha)$ have no common zeros in $\mathbb{P}^M(\bar{k})$ for all $\alpha\in A.$ 
We apply the following version of an effective Hilbert's Nullstellensatz.
\begin{lemma}\label{l:25}
 Let $L$ be an arbitrary field, $P_0, P_1,\ldots, P_l$ be homogeneous polynomials in $L(X_0,\ldots, X_M)$ of degree at most $d$ such that $P_0$ vanishes at all common zeros (if any) of $P_1,\ldots, P_l$ in the algeraic closure of $L.$ Then there exist  a positive integer $u\leq (4d)^{M+1}$ and homogeneous polynomials $Q_1,\ldots, Q_l$
such that $P_0^u=P_1Q_1+\ldots+P_lQ_l.$
\end{lemma}
Then, for each $ 0\leq i \leq M$, there exist  an integer $$r_i(\alpha)\leq (4\max_{i,l} \{\deg P_l, \deg Q_i\})^{M+1}$$
and homogeneous polynomials $\tilde{A}_{il}(\alpha)$ of degree $r_i(\alpha)-\deg P_l, \tilde{B}_{ij}(\alpha)$ of degree $r_i(\alpha)-\deg Q_j$ in $k[X_0,\ldots, X_M]$ ($ 1\leq l \leq r$, $ 0\leq j \leq m$)
 such that
$$ X_i^{r_i(\alpha)}=\sum_{l=1}^r\tilde{A}_{il}(\alpha)P_l+\sum_{j=0}^m\tilde{B}_{ij}(\alpha)Q_j(\alpha).$$
Since $r_i(\alpha)$ is bounded, by passing to infinite subset of $A$, we can assume that  $r_i(\alpha)$ is a constant denoted by $r_i.$ We rewrite the above equation, 
$$ a_i(\alpha)X_i^{r_i}=\sum_{l=1}^r\tilde{A}_{il}(\alpha)P_l+\sum_{j=0}^m\tilde{B}_{ij}(\alpha)Q_j(\alpha),$$
where $a_i: A \longrightarrow k, a_i(\alpha)=1$ for all $\alpha.$
We may regard the above equation as a system of linear equations in coefficients of $\tilde{A}_{il}(\alpha), 1 \leq l\leq r; \tilde{B}_{ij}(\alpha), 0\leq j \leq m,$ and $ a_i(\alpha)$. Thus, we can apply Lemma \ref{l:22} to this system.
 We  choose $A_{il}=\frac{\tilde{A}_{il}}{a_i},$ $B_{ij}=\frac{\tilde{B}_{ij}}{a_i} ( 0\leq i \leq M, , 1\leq l \leq r,  0\leq j \leq m)$.
\end{proof}

Let $\Lambda$ be an infinite index set. Let $x$ be a map $x: \Lambda\longrightarrow\mathbb{P}^M(k)$. A map $(C,a)\in\mathcal{R}^0_{\Lambda}$ is called \emph{small} with respect to $x$ iff 
$$h(a(\alpha))=o(h(x(\alpha))).$$
 Denote by $\mathcal{K}_x$ the set of all "small" maps. Then, $\mathcal{K}_x$ is a subring of $\mathcal{R}^0_{\Lambda}.$ Furthermore, if $(C,a)\in\mathcal{K}_x$ and $a(\alpha)\not=0$ for all but finitely $\alpha\in C$ then we have $\left(C\backslash\{a(\alpha)=0\}, \dfrac{1}{a}\right)\in\mathcal{K}_x$.

Denote by $\mathcal{C}_{x}$ the set of all positive functions $h$ defined over $\Lambda$ outside a finite subset of $\Lambda$  such that 
$$\log^+(h(\alpha))=o(h(x(\alpha))).$$

Then, $\mathcal{C}_{x}$ is a ring. Moreover, if $(C,a)\in\mathcal{K}_x$ then for every $v\in M_k$, the function $\|a\|_v: C\longrightarrow \mathbb{R}^+$ given by $\alpha\longmapsto \|a(\alpha)\|_v$ belongs to $\mathcal{C}_{x}.$ Furthermore, if $(C,a)\in\mathcal{K}_x, a(\alpha)\not=0$ for all but finitely many $\alpha\in C$ then the function $h:\{\alpha| a(\alpha)\not=0\}\longmapsto\dfrac{1}{\|a(\alpha)\|_v}$ also lies in $ \mathcal{C}_{x}$. 
We have the following lemma:
\begin{lemma}\label{l:26} Let the assumption be as in Lemma \ref{l:24}. We further assume that $Q_0,\ldots, Q_m$ are of the same degree $d$. Let  $x: \Lambda \longrightarrow V(k)$ be a map. 
Then, for  every $v\in M_k$, there exist  functions $l_{1,v}, l_{2,v}$ such that
$$l_{2,v}(\alpha)\|x(\alpha)\|_v^d\leq \max_{0\leq j \leq m}\|Q_j(x(\alpha))\|_v\leq l_{1,v}(\alpha)\|x(\alpha)\|_v^d,$$
for all $\alpha\in A$ outside a finite subset of A. Moreover, if the coefficients of $Q_j, j=0,\ldots, m$ belong to $\mathcal{K}_x$ then $l_{1,v}, l_{2,v}\in\mathcal{C}_x$.
\end{lemma}
\begin{proof}
 
It is easy to see that
\begin{align}\label{e:22}\|Q_j(x(\alpha))\|_v\leq \epsilon^{n_v}_v(|\mathcal{T}_d|)\|Q_j(\alpha)\|_v\|x(\alpha)\|_v^d.
\end{align}
Set $l_{1,v}(\alpha)=\epsilon^{n_v}_v(|\mathcal{T}_d|)\sum_{j=0}^m\|Q_j(\alpha)\|_v, \alpha\in A.$
From \eqref{e:22}, we get
$$\|Q_j(x(\alpha))\|_v\leq l_{1,v}(\alpha)\|x(\alpha)\|^d,  0\leq  j\leq m, \alpha\in A.$$
Since the coefficients of $Q_j, 0\leq j\leq m$ belong to $\mathcal{K}_x$ and $\mathcal{C}_x$ is a ring, we have $l_{1,v}\in\mathcal{C}_x.$ 

Now, we will prove the left-hand side inequality.  Applying Lemma \ref{l:24}, we have: for each $0\leq i\leq M,$ there exist a positive integer $r_i$ and  homogeneous polynomials $A_{il}$ of degree $r_i-\deg P_l$, $B_{ij}$ of degree $r_i-d$ in $ \mathcal{R}_{A,\{Q_j\}_{j=0}^m}[X_0,\ldots, X_M] (1\leq l \leq r, 0\leq j \leq m)$ such that
$$X_i^{r_i}=\sum_{l=1}^rA_{il}P_l+\sum_{j=0}^mB_{ij}Q_j.$$
Hence 
$$ x_i^{r_i}(\alpha)=\sum_{j=0}^mB_{ij}(x(\alpha))Q_j (x(\alpha)).$$
Therefore, we have for all $ \alpha\in A,$
$$\|x_i(\alpha)^{r_i}\|_v\leq \epsilon_v^{n_v}(m+1)\epsilon_v^{n_v}(|\mathcal{T}_{r_i-d}|)\max_{0\leq j\leq m}\|Q_j(x(\alpha)\|_v\\
\sum_j\|B_{ij}(\alpha)\|_v\|x(\alpha)\|^{r_i-d}_v .$$
Since $\|x(\alpha)\|_v=\max(\|x_0(\alpha)\|_v,\ldots,\|x_M(\alpha)\|_v)$ then there exists $i\in \{0,\ldots, M\}$ such that $\|x(\alpha)\|_v=\|x_i(\alpha)\|_v$. Hence,
$$\|x(\alpha)\|_v^d\leq  \max_iC_i\max_{0\leq j\leq m}\|Q_j(x(\alpha))\|_v \sum_{j,i}\|B_{ij}(\alpha)\|_v,$$
where $C_i=\epsilon_v^{n_v}((m+1)|\mathcal{T}_{r_i-d}|)$.
Set $$l_{2,v}(\alpha)=\dfrac{1}{ \max_i C_i\left( \sum\limits_{i, j}\|B_{ij}(\alpha)\|_v,\right)}.$$
  Since the coefficients of $B_{ij}$ belong to $\mathcal{R}_{A,\{Q_j\}}$, they vanish for all $\alpha$ or they vanish for only finitely many $\alpha \in A$. Hence, $l_{2,v}$ is defined outside a finite subset of $A$ and
$$l_{2,v}(\alpha)\|x(\alpha)\|_v^d\leq\max_{0\leq j\leq m}\|Q_j(x(\alpha))\|_v$$
for all $\alpha\in A$ outside a finite subset of $A$. Since $ a_{j,I}\in\mathcal{K}_x$, we have $\mathcal{R}_{A,\{Q_j\}}\subset \mathcal{K}_x$. Therefore $l_{2,v}\in\mathcal{C}_{x}$.
\end{proof}
For each positive integer $N$, let $k[X_0,\ldots, X_M]_N$ denote $k-$ vector space of homogeneous polynomials in $k[X_0,\ldots, X_M]$ of degree $N$ (including 0). Put
$$ (I_V)_N= k[X_0,\ldots, X_M]_N\cap I_V.$$  
Then, the Hilbert function of $V$ is defined by
$$ H_V(N)=\dim_k (k[X_0,\ldots, X_M]_N/(I_V)_N). $$
By the usual theory of Hilbert polynomials, we have
\begin{align}\label{e:203}
H_V(N)=\triangle\cdot \dfrac{N^n}{n!}+O(N^{n-1})\;\text{as}\; N\longrightarrow \infty,
\end{align}
where $\triangle$ is the projective degree of $V.$

 Let $\mathcal{R}_{A,\{Q_j\}_{j=1}^q}[X_0,\ldots, X_M] _N$) denote $\mathcal{R}_{A,\{Q_j\}_{j=1}^q}-$vector space of homogeneous polynomials in $\mathcal{R}_{A,\{Q_j\}_{j=1}^q}[X_0,\ldots, X_M]$ of degree $N$ (including 0). Put
$$ (I_{A, V,\{Q_j\}_{j=1}^q})_N=\mathcal{R}_{A,\{Q_j\}_{j=1}^q}[X_0,\ldots, X_M] _N\cap I_{A, V,\{Q_j\}_{j=1}^q}. $$
We will prove that
\begin{lemma}\label{l:27}
$ \dim_{ \mathcal{R}_{A,\{Q_j\}_{j=1}^q}}\left( \mathcal{R}_{A,\{Q_j\}_{j=1}^q}[X_0,\ldots, X_M]_N/(I_{A, V, \{Q_j\}_{j=1}^q})_N\right)= \triangle\cdot \dfrac{N^n}{n!}+O(N^{n-1})$\\
$\text{as}\; N\longrightarrow \infty.$
\end{lemma}
\begin{proof}
Let $\phi_1,\ldots, \phi_p$ be  monomials. We will prove that $\phi_1,\ldots, \phi_p$ are linearly independent in $k-$vector space $k[X_0,\ldots, X_M]_N/(I_V)_N$ if and only if $\phi_1,\ldots, \phi_p$ are linearly independent in $\mathcal{R}_{A,\{Q_j\}_{j=1}^q}$-vector space $\mathcal{R}_{A,\{Q_j\}_{j=1}^q}[X_0,\ldots, X_M]_N/(I_{A, V, \{Q_j\}_{j=1}^q})_N$.

The part " if" is obvious. We prove the part " only if". Assume that $\phi_1,\ldots, \phi_p$ are linearly dependent in $\mathcal{R}_{A,\{Q_j\}_{j=1}^q}$-vector space $\mathcal{R}_{A,\{Q_j\}_{j=1}^q}[X_0,\ldots, X_M]_N/(I_{A, V, \{Q_j\}_{j=1}^q})_N$. Then,
  there exist $a_i\in \mathcal{R}_{A,\{Q_j\}_{j=1}^q} $ not all $0$ such that
$$ \sum_{i=1}^{p} a_i\phi_i\in (I_{A, V, \{Q_j\}_{j=1}^q})_N. $$
Therefore $\sum_{i=1}^{p} a_i(\alpha)\phi_i \in (I_V)_N.$ Since $\phi_1,\ldots, \phi_p$ are linearly independent in $k-$vector space $k[X_0,\ldots, X_M]_N/(I_V)_N$, we have $a_i(\alpha)=0$ for all  $\alpha \in A$ and $i=1,\ldots, p.$ We got a contradiction. Hence, the claim is true. Together with \eqref{e:203}, it completes the proof.
\end{proof}
\begin{lemma}\label{l:28}
  Let  $\phi_1,\ldots, \phi_{H_V(N)}$ is a monomial basis of $\mathcal{R}_{A,\{Q_j\}_{j=1}^q}[X_0,\ldots, X_M]_N/(I_{A, V, \{Q_j\}_{j=1}^q})_N$. Then, for each homogeneous polynomial $Q\in \mathcal{R}_{A,\{Q_j\}_{j=1}^q}[X_0,\ldots, X_M]_N $ with one coefficient equal to 1, there exists a linear form $L=\sum_ia_iY_i\in \mathcal{R}_{A,\{Q_j\}_{j=1}^q}[Y_1,\ldots, Y_{H_V(N)}] $ such that
$$ Q\equiv L (\phi_1,\ldots, \phi_{H_V(N)})\; (\text{mod}\; I_{A, V,\{Q_j\}_{j=1}^q}).$$
Furthermore, for all  $\alpha\in A$,
$$h(a_i(\alpha))\leq 2C\left(h(Q(\alpha))+\sum_{j=1}^rh(P_j)+\log (C)\right), 1\leq i\leq H_V(N),$$
$$ h(L(\alpha))\leq 2C\left(h(Q(\alpha))+\sum_{i=1}^rh(P_i)+\log (C)\right), $$
where $C=(3rN^M)^pp^2, p=H_V(N)+r|\mathcal{T}_N|+1.$
\end{lemma}
\begin{proof}
Without loss of generality, we can assume that $P_1,\ldots, P_r$ have leading coefficients 1.
Since $\phi_1,\ldots, \phi_{H_V(N)}$ is a monomial basis of $\mathcal{R}_{A,\{Q_j\}_{j=1}^q}[X_0,\ldots, X_M]_N/(I_{A, V, \{Q_j\}_{j=1}^q})_N$, there exist $\tilde{a}_i\in \mathcal{R}_{A,\{Q_j\}_{j=1}^q}, i=1,\ldots, H_V(N)$ such that
$$ Q\equiv \sum_{i=1}^{H_V(N)} \tilde{a}_i\phi_i  \; (\text{mod}\; I_{A, V,\{Q_j\}_{j=1}^q}).$$
Therefore, there exist homogeneous polynomials $B_i\in \mathcal{R}_{A,\{Q_j\}_{j=1}^q}[X_0,\ldots, X_M]$ of degree $N-\deg P_i$, ($ i=1,\ldots, r$) such that
$$ Q=\sum_{i=1}^r B_i P_i+ \sum_{i=1}^{H_V(N)} \tilde{a}_i\phi_i.$$
Rewrite the above equation in the formula
$$ aQ=\sum_{i=1}^r B_i P_i+ \sum_{i=1}^{H_V(N)} \tilde{a}_i\phi_i, $$
where $a\in  \mathcal{R}_{A,\{Q_j\}_{j=1}^q} $, $a(\alpha)=1$ for all $\alpha.$ We may consider the above equation as a system of linear equations in coefficients of $B_i, 1\leq i \leq r; \tilde{a}_i, 1\leq i\leq H_V(N)$ and $a$. The number of variables does not exceed $ H_V(N)+r|\mathcal{T}_N|+1:=p$. The coefficients of this system are $\mathbb{Z}-$combinational of at most $rN^M+2$ coefficients of $Q, P_i.$ Applying Lemma \ref{l:22}, there exist $B_i\in\mathcal{R}_{A,\{Q_j\}_{j=1}^q}[X_0,\ldots, X_M], 1\leq i \leq r; \tilde{a}_i, 1\leq i\leq H_V(N), 0\not=a\in \mathcal{R}_{A,\{Q_j\}_{j=1}^q}$ satisfying the above equation and
$$ \|a(\alpha)\|,\|\tilde{a}_i(\alpha)\| _v\leq \epsilon_v^{n_v}(p^2)\left[\epsilon_v^{n_v}(rN^M+2)\max_i(\|Q(\alpha)\|_v, \|P_i\|_v,1)\right]^p.$$
Denote by $C=(3rN^M)^pp^2$. Since $Q, P_j$ have at least one coefficient equal to 1, we have $\|Q(\alpha)\|_v, \|P_j\|_v\geq 1$ for all $v\in M_k$. Therefore
\begin{align}\label{e:2001}
 \max(\|a(\alpha)\| _v,\|\tilde{a}_i(\alpha)\| _v, 1)\leq \epsilon_v^{n_v}(C) \|Q(\alpha)\|_v^C\prod_{j=1}^r \|P_j\|_v^C.
\end{align}
Hence
\begin{align}\label{e:2002}
h(a(\alpha)), h(\tilde{a}_i(\alpha))\leq C(h(Q(\alpha))+\sum_{j=1}^{r}h(P_j)+\log C) .
\end{align}
Set $\tilde{L}=\sum_{i=1}^{H_V(N)} \tilde{a}_iY_i \in \mathcal{R}_{A,\{Q_j\}_{j=1}^q}[Y_1,\ldots,Y_{H_V(N)}].$ Then, from \eqref{e:2001}, we have
$$ \|\tilde{L}(\alpha)\|_v \leq \epsilon_v^{n_v}(C) \|Q(\alpha)\|_v^C\prod_{j=1}^r \|P_j\|_v^C. $$
Therefore, we have 
\begin{align}\label{e:2005}
h(\tilde{L}(\alpha))\leq C(h(Q(\alpha))+\sum_{j=1}^{r}h(P_j)+\log C).
\end{align}
Set $a_i=\dfrac{\tilde{a}_i}{a}$ and  $L=\sum_ia_iY_i$. Then $$Q\equiv L(\phi_1,\ldots, \phi_{H_V(N)})\;\text{mod}\; I_{A,V,\{Q_j\}_{j=1}^q}.$$ 
Furthermore, from \eqref{e:2002} and \eqref{e:2005}, we have
 \begin{align*}
h(a_i(\alpha))\leq h(\tilde{a}_i(\alpha))+h(a(\alpha))\leq 2C\left( h(Q(\alpha))+\sum_{j=1}^{r}h(P_j)+\log C\right)
\end{align*}
and
$$ h(L(\alpha)) =h\left(\dfrac{1}{a(\alpha)}\tilde{L}(\alpha)\right)\leq h(a(\alpha))+h(\tilde{L}(\alpha))\leq 2C\left( h(Q(\alpha))+\sum_{j=1}^{r}h(P_j)+\log C\right),$$
which completes the proof.

\end{proof}
\begin{lemma}\label{l:29}
 Let $\phi_1,\ldots, \phi_{H_V(N)}$ be a monomial basis of $\mathcal{R}_{A,\{Q_j\}_{j=1}^q}[X_0,\ldots, X_M]_N/(I_{A, V, \{Q_j\}_{j=1}^q})_N$. We define
$$ F=[\phi_1:\ldots: \phi_{H_V(N)}] .$$
Let $x: \Lambda \longrightarrow V(k)$ be a collection of points. Then
$$ h(F(x(\alpha)))=Nh(x(\alpha))+O(1), $$
where $O(1)$ is a constant depending only on $V,  N.$
\end{lemma}
\begin{proof}
It is obvious that 
$$ \|F(x(\alpha))\|_v=\max_i\|\phi_i(x(\alpha))\| _v\leq \|x(\alpha)\|_v^N.$$
Therefore 
\begin{align}\label{e:2007}
h(F(x(\alpha)))\leq N h(x(\alpha)).
\end{align}
For each $0\leq i\leq M,$ we apply Lemma \ref{l:28} to $X_i^N$,  there exists a linear form  $L_i=\sum_ja_{ij}Y_j\in \mathcal{R}_{A,\{Q_j\}_{j=1}^q}[Y_1,\ldots, Y_{H_V(N)}]$ such that
$$ X_i^N\equiv L_i(\phi_1,\ldots, \phi_{H_V(N)}) \;\;\;(\text{mod}\; I_{A, V,\{Q_j\}_{j=1}^q}). $$
Furthermore, for all $\alpha \in A,$
\begin{align}\label{e:2003}
h(a_{ij}(\alpha))\leq 2C(\sum_{i=1}^{r}h(P_i)+\log C), 1\leq j\leq H_V(N),0\leq i\leq M.
\end{align}
Hence
$$ x_i^N(\alpha)= \sum_{j=1}^{H_V(N)}a_{ij}(\alpha)\phi_j (x(\alpha))$$
Therefore, for every $v\in M_k,$
\begin{align*} \|x_i(\alpha)\|^N_v&\leq \epsilon_v^{n_v}(H_V(N))\max_j\|a_{ij}(\alpha)\|_v\|F(x(\alpha))\|_v \\
&\leq \epsilon_v^{n_v}(H_V(N))\max_{i,j}\|a_{ij}(\alpha)\|_v\|F(x(\alpha))\|_v
\end{align*}
for all $0\leq i\leq M.$ Hence,
\begin{align*}
\|x(\alpha)\|^N_v&\leq \epsilon_v^{n_v}(H_V(N))\max_{i,j}\|a_{ij}(\alpha)\|_v\|F(x(\alpha))\|_v\\
&\leq \epsilon_v^{n_v}(H_V(N))\prod_{i,j}\max(\|a_{ij}(\alpha)\|_v,1)\|F(x(\alpha))\|_v.
\end{align*}
Together with \eqref{e:2003} , we have
\begin{align}\label{e:2006}
  Nh(x(\alpha))&\leq h(F(x(\alpha)))+\sum_{i=0}^M\sum_{j=1}^{H_V(N)}h(a_{ij})+\log H_V(N),\notag\\
&\leq h(F(x(\alpha)))+2(M+1)H_V(N)C(\sum_{i=1}^{r}h(P_i)+\log C)+\log H_V(N).
   \end{align}
From \eqref{e:2007} and \eqref{e:2006}, we have the desired result.


\end{proof}
\section{Proof of The Main Theorem}
We will show that the Main Theorem is an implication of the following theorem.
\begin{theorem}\label{t:31} Let k be a number field, $M_k^\infty\subset S$ be a finite set of places of k, let $q, m, n$ be  positive integers with $q>m\geq n$ and $\epsilon>0.$ Let $\Lambda$ be an infinite index set, let $Q_1,\ldots, Q_q$ be moving hypersurfaces in $\mathbb{P}^M(k)$ respectively of the same degree d. Let $V$ be an irreducible projective subvariety of $\mathbb{P}^M$ defined over $k$ of dimension $n$ and degree $\triangle.$ Let $x: \Lambda\longrightarrow V(k)$ be a collection of points such that:

(1) The family of polynomials $Q_1,\ldots, Q_q$ locates in m-subgeneral position with respect to $V$;
 
 (2)  x is algebraically non-degenerate over $V$ and  $\mathcal{R}_{\{Q_j\}_{j=1}^q}$;
 
 (3) $ h(Q_j(\alpha))=o(h(x(\alpha)))$ for all  $j=1,\ldots, q.$\\
 Let $A\subset \Lambda$ be an infinite index subset which is coherent with respect to $\{Q_j\}_{j=1}^q$. Suppose moreover that:
 
 $ \bullet$ all the polynomials $Q_j$ have coefficients in $\mathcal{R}_{A, \{Q_j\}_{j=1}^q}.$ Furthermore, for each $1\leq j\leq q,$  $ Q_j$ has one coefficient equal to 1.
 
 $\bullet$ the polynomials $Q_j's$ never vanish at x over $A$, i.e. for any $\alpha\in A$ and any $j=1,\ldots, q,$ we have
 $$Q_j(\alpha)(x(\alpha))\not=0.$$ 
 
 Then there exists an infinite index subset $B\subset A \subset\Lambda$ such that
$$\sum_{v\in S}\sum_{j=1}^q\dfrac{1}{d}\lambda_{Q_j(\alpha),v}(x(\alpha))\leq (m(n+1)+\epsilon)h(x(\alpha)),$$
for all $\alpha\in B.$
\end{theorem}
\begin{proof} 
Fix $v\in S$. Given $\alpha\in A,$ there exists a renumbering $\{j_1(v,\alpha),\ldots, j_q(v,\alpha)\}$ of the indices $\{1,\ldots,q\}$ such that
$$ |Q_{j_1(v,\alpha)}(x(\alpha))|_v\leq \cdots \leq |Q_{j_m(v,\alpha)}(x(\alpha))|_v\leq \min_{j\not\in \{j_1(v,\alpha),\ldots, j_m(v,\alpha)\}} Q_{j}(x(\alpha))$$
for all $\alpha \in A$. 

By Lemma \ref{l:26}, we have
\begin{align}\label{e:36}
\log\prod_{j=1}^q\|Q_j(x(\alpha))\|_v=\log\prod_{i>m}\|Q_{j_i}(x(\alpha))\|_v+\log\prod_{i=1}^m\|Q_{j_i(v)}(x(\alpha))\|_v\notag\\
\geq d(q-m) \log\|x(\alpha)\|_v-\log\tilde{h}_v(\alpha)+m\log\|Q_{j_1(v,\alpha)}(x(\alpha))\|_v,
\end{align}
where $\tilde{h}_v$ as a product of the form $\prod\left(1+\dfrac{1}{h_{\mu}}\right)$ ( $h_{\mu}$ run over all the choices of $l_{2,v}$). Since for each $j=1,\ldots, q$, $Q_j$ has at least one coefficient equal to 1 and $h(Q_j(\alpha))=o(h(x(\alpha)))$,  we have the coefficients of $Q_j, 1\leq  j\leq q,$ 
belong to $\mathcal{K}_x.$ Therefore, $\tilde{h}_v\in\mathcal{C}_x.$

By \eqref{e:36}, for all $\alpha\in A,$
 \begin{align}\label{e:37}
\log\prod_{j=1}^q\dfrac{\|x(\alpha)\|_v^d}{\|Q_j(x(\alpha))\|_v}\leq m \log\dfrac{\|x(\alpha)\|_v^d}{\|Q_{j_1(v,\alpha)}(x(\alpha))\|_v}+\log\tilde{h}_v(\alpha).
\end{align}
For each integer $N$, put $$V_N=\mathcal{R}_{A,\{Q_j\}_{j=1}^q}[X_0,\ldots,X_M]_N/(I_{A, V,\{Q_j\}_{j=1}^q})_N.$$ 
For every positive integer $N$ with $d|N$, we consider the following filtration on the vector space $V_N$ with respect to $Q_{j_1(v,\alpha)}$: The filtration
$$V_N=W_0\supset W_1\supset\cdots\supset W_{N/d}$$
is defined by
$$W_i=\{g^*|g\in \mathcal{R}_{A,\{Q_j\}_{j=1}^q}[X_0,\ldots,X_M]_N\,\,\text{and}\,\, Q_{j_1(v,\alpha)}^i|g\},$$
where $g^*$ is the projection of $g$ to $\mathcal{R}_{A,\{Q_j\}_{j=1}^q}[X_0,\ldots,X_M]_N/(I_{A, V,\{Q_j\}_{j=1}^q})_N.$ Take a basis $\psi_1,\ldots, \psi_{H_V(N)}$ of the vector space $V_N$ compatible with the filtration $W_i$, by this, we mean that, for each $i=0,\ldots, N$, it contains a basis of $W_i$.

Furthermore, we can choose a basis $\psi_1,\ldots,\psi_{H_V(N)}$ such that
\begin{align}\label{e:41}
\psi_j=Q_{j_1(v,\alpha)}^{i_j}g_j,
\end{align}
where $g_j\in \mathcal{R}_{A,\{Q_j\}_{j=1}^q}[X_0,\ldots, X_M] _{N-di_j}$ is a monomial and $g_j$ does not divide $Q_{j_1(v,\alpha)}$. 
Hence,
\begin{align}\label{e:34}
\sum_{j=1}^{H_V(N)}\log \|\psi_j(x(\alpha))\|_v&\leq \left(\sum i_j \right)\log \|Q_{j_1(v,\alpha)}(x(\alpha))\|_v\notag\\
&+\left(NH_V(N)-d\left(\sum i_j\right)\right)\log \|x(\alpha)\|_v.
\end{align}
 Now, we estimate the sum $\sum_{j=1}^{H_V(N)}i_j.$ To do it, we modify a Lemma from \cite{CRY2}, Lemma 2.2. 
\begin{lemma}\label{l:3}(\cite{CRY2}, Lemma 2.2) $$\sum_{j=1}^{H_V(N)}i_j=\dfrac{\triangle N^{n+1}}{(n+1)!}(1+o(1)),$$ where the function $o(1)$ depends only on the variety $V$.
\end{lemma}
\begin{proof}
 It is clear that there are exactly $\dim (W_i/W_{i+1})$ elements $\psi_j$ with $i_j=i$ in the set $\psi_1,\ldots, \psi_{H_V(N)}.$ Hence,
\begin{align}\sum_{j=1}^{H_V(N)}i_j=\sum_{i=1}^{N/d} i\cdot\dim(W_i/W_{i+1}),\label{e:42}
\end{align}
in which $W_{N/d+1}=\{\vec{0}\}.$

Next, we claim that $\dim W_i=\dim V_{N-di}.$ To see it, notice that each element $\psi$ of $W_i$ can be represented as $\psi=Q_{j_1(v,\alpha)}^{i}.g$ with $g\in \mathcal{R}_{A, \{Q_j\}_{j=1}^q}[X_0,\ldots, X_M]_{N-di}.$ Furthermore, two polynomials $g_1, g_2$ such that $Q_{j_1(v,\alpha)}^{i}\cdot g_1=Q_{j_1(v,\alpha)}^{i}\cdot g_2$ in $W_i$ if and only if $Q_{j_1(v,\alpha)}^{i}(g_1-g_2)\in I_{A, V, \{Q_j\}_{j=1}^q}$. Therefore $$Q_{j_1(v,\alpha)}^{i}(\beta)(g_1(\beta)-g_2(\beta))\in I_{V}  $$  for all $\beta \in A$. Since $I_V$ is a prime ideal, we have   $g_1(\beta)-g_2(\beta)\in I_V$ for all $\beta\in A$. Applying Lemma \ref{l:23}, we have $g_1-g_2\in I_{A, V, \{Q_j\}_{j=1}^q}.$ Therefore $\dim W_i=\dim V_{N-di}=H_V(N-di).$ By Lemma \ref{l:27},  for each positive integer $L$,
$$\dim V_L=\triangle\cdot\dfrac{L^n}{n!}+O(L^{n-1}).$$
  Hence,
\begin{align*}
\sum_{i=1}^{N/d}i\cdot\dim (W_i/W_{i+1})&=\sum_{i=1}^{N/d} i(\dim W_i-\dim W_{i+1}) \\
&=\sum_{i=1}^{N/d}i\cdot\dim W_i-\sum_{i=1}^{N/d}((i+1)\dim W_{i+1}-\dim W_{i+1})\\
&=\sum_{i=1}^{N/d}\dim W_i=\dfrac{\triangle}{n!}\sum_{i=1}^{N/d}\left((N-di)^n+O(N^{n-1})\right)\\
&=\dfrac{\triangle N^n}{n!}\sum_{i=1}^{N/d}\left(1-\frac{di}{N}\right)^n+O(N^n)\notag\\
&=\dfrac{\triangle N^{n+1}}{n!d}\left(\int\limits_0^1(1-t)^n\cdot dt+o(1)\right)+O(N^n).
\end{align*}
Therefore
\begin{align}
\sum_{i=1}^{N/d}i\cdot\dim (W_i/W_{i+1})=\dfrac {\triangle N^{n+1}}{(n+1)!}(1+o(1)).\label{e:43}
\end{align}
Combining \eqref{e:42} and \eqref{e:43}, we obtain the desired result.
\end{proof}
Let $\phi_1,\ldots,\phi_{H_V(N)}$ be a fixed monomial basis of $V_N$. Set 
\begin{align*}F: V&\longrightarrow \mathbb{P}^{H_V(N)-1}\\
x&\longmapsto (\phi_1(x):\ldots:\phi_{H_V(N)}(x))
\end{align*}
Applying Lemma \ref{l:28}, there exist linear forms $L_j=\sum_{r=1}^{H_V(N)}w_{jr}Y_r$ with coefficients in $\mathcal{R}_{A,\{Q_j\}_{j=1}^q}, 1\leq j\leq H_V(N)$ such that  
$$\psi_j\equiv L_j(\phi_1,\ldots,\phi_{H_V(N)}),\; (\text{mod}\; I_{A, V, \{Q_j\}_{j=1}^q})\; 1\leq j\leq H_V(N).$$
Furthermore, for all $\alpha\in A,$ we have
\begin{align}\label{e:301}
 h(L_j(\alpha))\leq 2C(h(\psi_j(\alpha))+\sum_{i=1}^rh(P_i)+\log (C)), 1\leq j\leq H_V(N).
 \end{align}
 and
 \begin{align}\label{e:302}
 h(w_{jr}(\alpha))\leq 2C(h(\psi_j(\alpha))+\sum_{i=1}^rh(P_i)+\log (C)), 1\leq j, r\leq H_V(N).
\end{align}
 Then, we have
\begin{align*}\log \|L_j(F(x(\alpha)))\|_v=\log\|\psi_j(x(\alpha))\|_v.
\end{align*}
Combining with \eqref{e:34}, we have
\begin{align*}\log\prod_{j=1}^{H_V(N)}\dfrac{\|F(x(\alpha))\|_v\|L_j(\alpha)\|_v}{\|L_j(F(x(\alpha)))\|_v}\geq -(\sum i_j)\log\|Q_{j_1(v,\alpha)}(x(\alpha))\|_v\\
+(d(\sum_j i_j)-NH_V(N))\log\|x(\alpha)\|_v+H_V(N)\log\|F(x(\alpha))\|_v+\log\prod_{j=1}^{H_V(N)}\|L_j(\alpha)\|_v\\
=\dfrac{\triangle N^{n+1}}{(n+1)!}(1+o(1))\log\dfrac{\|x(\alpha)\|_v^d}{\|Q_{j_1(v,\alpha)}(x(\alpha))\|_v}+H_V(N)\log\|F(x(\alpha))\|_v\\
-NH_V(N)\log\|x(\alpha)\|_v
+\log\prod_{j=1}^{H_V(N)}\|L_j(\alpha)\|_v.
\end{align*}
This is equivalent to
\begin{align}\label{e:310}
\dfrac{\triangle N^{n+1}}{(n+1)!}(1+o(1))\log \dfrac{\|x(\alpha)\|_v^d}{\|Q_{j_1(v,\alpha)}(x(\alpha))\|_v}\leq \log\prod_{j=1}^{H_V(N)}\dfrac{\|F(x(\alpha))\|_v\|L_j(\alpha)\|_v}{\|L_j(F(x(\alpha)))\|_v}\notag\\
-H_V(N)\log\|F(x(\alpha))\|_v
+NH_V(N)\log\|x(\alpha)\|_v-\log\prod_{j=1}^{H_V(N)}\|L_j(\alpha)\|_v.
\end{align}
 
 \begin{lemma}\label{l:32} We have, for all $v\in M_k$,
 
 i) $h(L_j(\alpha))=o(h(F(x(\alpha)))), j=1,\ldots, H_V(N).$
 
 ii) $-\log\prod\limits_{j=1}^{H_V(N)}\|L_j(\alpha)\|_v\leq
2NH_V(N)C\left(\sum\limits_{j=1}^qh(Q_j(\alpha))+\sum\limits_{i=1}^r h(P_i)+\log C\right).$
 \end{lemma}
\begin{proof}
From \eqref{e:41}, we have $h(\psi_j(\alpha))\leq i_jh(Q_{j_1(v,\alpha)}(\alpha))$ with $i_j\leq N/d.$
 Combining with \eqref{e:301},\eqref{e:302}, we have
 \begin{align}\label{e:303}
 h(L_j(\alpha)), h(w_{jr}(\alpha))\leq 2C\left(\dfrac{N}{d}h(Q_{j_1(v,\alpha)}(\alpha))+\sum_{i=1}^rh(P_i)+\log (C)\right).
 \end{align}
 Applying Lemma \ref{l:29}, we have
 \begin{align}\label{e:304}
 h(F(x(\alpha)))=Nh(x(\alpha))+O(1).
 \end{align}
 Combining \eqref{e:303}, \eqref{e:304} with the fact that $h(Q_j(\alpha))=o(h(x(\alpha)))$, we have the desired result in i).
 Part ii) follows from the following inequality
 \begin{align*}
 \log \|L_j(\alpha)\|_v&\geq \log\|w_{jr}(\alpha)\|_v\geq -h(w_{jr}(\alpha))\\
 &\geq -2C\left(\dfrac{N}{d}h(Q_{j_1(v,\alpha)}(\alpha))+\sum_{i=1}^rh(P_i)+\log (C)\right),
 \end{align*}
 where the last inequality follows from \eqref{e:303}.

\end{proof}

We continue to prove Theorem \ref{t:31}.

From \eqref{e:310} and applying Lemma \ref{l:32} ii, we have
\begin{align}\label{e:3003}
\dfrac{\triangle N^{n+1}}{(n+1)!}(1+o(1))\log \dfrac{\|x(\alpha)\|_v^d}{\|Q_{j_1(v,\alpha)}(x(\alpha))\|_v}\leq \log\prod_{j=1}^{H_V(N)}\dfrac{\|F(x(\alpha))\|_v\|L_j(\alpha)\|_v}{\|L_j(F(x(\alpha)))\|_v}\notag\\
-H_V(N)\log\|F(x(\alpha))\|_v
+NH_V(N)\log\|x(\alpha)\|_v\notag\\
+2NH_V(N)C\left(\sum\limits_{j=1}^qh(Q_j(\alpha))+\sum\limits_{i=1}^r h(P_i)+\log C\right).
\end{align}
Notice that the collection of all possible linear forms $L_j$, $1\leq j\leq H_V(N)$ when $ v$ runs over $ S$ and $ \alpha$ runs over $ A$ is a finite set and denote it by $\{L_1,\ldots, L_u\}.$ 

Combining \eqref{e:37}, \eqref{e:3003}  taking sum over $v\in S$, we have
 \begin{align}\label{e:316}
\dfrac{\triangle N^{n+1}}{(n+1)!m}(1+o(1))\sum_{v\in S}\log\prod_{j=1}^q\dfrac{\|x(\alpha)\|_v^d}{\|Q_j(x(\alpha))\|_v}\notag\\
\leq
\sum_{v\in S} \left(\max_K\log\prod_{j\in K}\dfrac{\|F(x(\alpha))\|_v\|L_j(\alpha)\|_v}{\|L_j(F(x(\alpha)))\|_v}-H_V(N)\log\|F(x(\alpha))\|_v\right)\notag\\
+NH_V(N)\sum_{v\in S}\log\|x(\alpha)\|_v+2NH_V(N)|S|C\left(\sum_jh(Q_j(\alpha))+\sum_{i=1}^r h(P_i)+\log C\right)\notag\\
+ \dfrac{1}{m}\sum_{v\in S}\log\tilde{h}_v(\alpha),
\end{align}
where the maximum is taken over all subsets $K$ of $\{1,\ldots, u\}$ such that $L_j, j\in K$ are linearly independent over $\mathcal{R}_{A,\{Q_j\}_{j=1}^q}$. Since the left-hand side of above inequality is independent of the choice of components of $x(\alpha)$, we can choose the components such that $\|x(\alpha)\|_v=1$ for all $v\not\in S$. So, we have
$$\sum_{v\in S}\log\|x(\alpha)\|_v=h(x(\alpha)), \sum_{v\in S}\log\|F(x(\alpha))\|_v\geq h(F(x(\alpha))).$$
Combining with \eqref{e:316}, we have
 \begin{align}\label{e:319}
\dfrac{\triangle N^{n+1}}{(n+1)!m}(1+o(1))\sum_{v\in S}\log\prod_{j=1}^q\dfrac{\|x(\alpha)\|_v^dv}{\|Q_j(x(\alpha))\|}\notag\\
\leq
\left(\sum_{v\in S}\max_K\log\prod_{j\in K}\dfrac{\|F(x(\alpha))\|_v\|L_j(\alpha)\|_v}{\|L_j(F(x(\alpha)))\|_v}-H_V(N)h(F(x(\alpha)))\right)+NH_V(N)h(x(\alpha))\notag\\
+2NH_V(N)|S|C\left(\sum_{j=1}^qh(Q_j(\alpha))+\sum_{i=1}^r h(P_i)+\log C\right)+\dfrac{1}{m}\sum_{v\in S}\log\tilde{h}_v(\alpha),
\end{align}
where the maximum is taken over all subsets $K$ of $\{1,\ldots, u\}$ such that $L_j, j\in K$ are linearly independent over $\mathcal{R}_{A,\{Q_j\}_{j=1}^q}$.

Now, we check the conditions $(1)-(2)$ of  Schmidt's subspace theorem for moving hyperplanes $L_j(\alpha), 1\leq j\leq H_V(N)$ and a collection of points $F(x(\alpha)): A\longrightarrow \mathbb{P}^{H_V(N)-1}(k)$.

 $\bullet$ Condition (1).  Since $A$ is coherent with respect to $Q_j, 1\leq j\leq q$ and the coefficients of $L_j, 1\leq j\leq u,$ are in $\mathcal{R}_{A,\{Q_j\}_{j=1}^q}$, the set  $A$ is also coherent with respect to $\{L_j\}_{j=1}^{u}$. Moreover, $\mathcal{R}_{A,\{Q_j\}_{j=1}^q}\supset\mathcal{R}_{A,\{L_j\}_{j=1}^{u}}.$ Since $x$ is algebraically non-degenerate over $\mathcal{R}_{A,\{Q_j\}_{j=1}^q}$, there is no homogeneous polynomial $Q\in \mathcal{R}_{A,\{Q_j\}_{j=1}^q}[X_0,\ldots,X_M]$ $\backslash I_{A, V,\{Q_j\}_{j=1}^q}$ such that $Q(x_o(\alpha),\ldots,x_M(\alpha))=0,$ for all $\alpha\in A$ outside a finite subset of $A$. Hence, the restrictions of coordinates  $y_1,\ldots,y_{H_V(N)}$ of the map $y:\Lambda\longrightarrow\mathbb{P}^{H_V(N)-1}(k), \alpha\longmapsto F(x(\alpha))=(y_1(\alpha),\ldots,y_{H_V(N)}(\alpha))$ to $A$ are linearly independent over $\mathcal{R}_{A,\{Q_j\}_{j=1}^q}\supset\mathcal{R}_{A,\{L_j\}_{j=1}^u}.$


$\bullet$ Condition (2). See lemma \ref{l:32} (i).

Now, we can apply theorem A for moving hyperplanes $L_j(\alpha), 1\leq j\leq H_V(N)$ and a collection of points $F(x(\alpha)): A\longrightarrow \mathbb{P}^{H_V(N)-1}(k)$. Then, there exists an infinite subset $B\subset A$ such that for all $\alpha\in B$,
\begin{align}\label{e:318}\sum_{v\in S}\max_K\log\prod_{j\in K}\dfrac{\|F(x(\alpha))\|_v\|L_j(\alpha)\|_v}{\|L_j(F(x(\alpha)))\|_v}\leq (H_V(N)+\epsilon) h(F(x(\alpha))),
\end{align}
where the maximum is taken over all subsets $K$ of $\{1,\ldots, u\}$ such that $L_j(\alpha), j\in K$ are linearly independent over $k$ for all $\alpha\in A.$

Notice that if $L_j, j\in K$ are linearly independent over $\mathcal{R}_{A,\{Q_j\}_{j=1}^q}$ then $L_j(\alpha), j\in K$ are linearly independent over $k$ for all but finitely $\alpha\in A$. Thus, combining \eqref{e:318} and \eqref{e:319}, we have
 \begin{align*}
\dfrac{\triangle N^{n+1}}{(n+1)!m}(1+o(1))\sum_{v\in S}\log\prod_{j=1}^q\dfrac{\|x(\alpha)\|_v^d}{\|Q_j(x(\alpha))\|_v}\leq
N\epsilon h(x(\alpha))+NH_V(N)h(x(\alpha))\notag\\
+2H_V(N)N|S|C\left(\sum_{j=1}^qh(Q_j(\alpha))+\sum_{i=1}^r h(P_i)+\log C\right)+\dfrac{1}{m}\sum_{v\in S}\log\tilde{h}_v(\alpha).
\end{align*}
Hence,
\begin{align}\label{e:320}
\sum_{v\in S}\log\prod_{j=1}^q\dfrac{\|x(\alpha)\|_v^d}{\|Q_j(x(\alpha))\|_v}\leq m(n+1)(1+o(1))h(x(\alpha))\notag\\
+2m(n+1)|S|C(1+o(1))
\left(\sum_{j=1}^qh(Q_j(\alpha))+\sum_{i=1}^r h(P_i)+\log C\right)\notag\\
+\dfrac{(n+1)!}{\triangle N^{n+1}}(1+o(1))\sum_{v\in S}\log\tilde{h}_v(\alpha).
\end{align}
Notice that $\sum_{v\in S}\log\|Q_j(\alpha)\|_v\leq h(Q_j(\alpha))$ (one coefficient of $Q_j$ equals to 1). Hence,
\begin{align}\label{e:321}
\sum_{v\in S}\log\prod_{j=1}^q\dfrac{\|x(\alpha)\|_v^d\|Q_j(\alpha)\|_v}{\|Q_j(x(\alpha))\|_v}
\leq \sum_{v\in S}\log\prod_{j=1}^q\dfrac{\|x(\alpha)\|_v^d}{\|Q_j(x(\alpha))\|_v}+ q\sum_{j=1}^qh(Q_j(\alpha)).
\end{align}
From \eqref{e:320} and \eqref{e:321}, we have
\begin{align}
\sum_{v\in S}\log\prod_{j=1}^q\dfrac{\|x(\alpha)\|_v^d\|Q_j(\alpha)\|_v}{\|Q_j(x(\alpha))\|_v} \leq m(n+1)(1+o(1))h(x(\alpha))\notag\\
+2m(n+1)|S|qC(1+o(1))
\left(\sum_{j=1}^qh(Q_j(\alpha))+\sum_{i=1}^r h(P_i)+\log C\right)\notag\\
+\dfrac{(n+1)!}{\triangle N^{n+1}}(1+o(1))\sum_{v\in S}\log\tilde{h}_v(\alpha)
\end{align}
Since $h(Q_j(\alpha))=o(h(x(\alpha))), \log\tilde{h}_v(\alpha)=o(h(x(\alpha)))$ and by Northcott's theorem, we have
\begin{align*}2m(n+1)|S|Cq(1+o(1))
\left(\sum_{j=1}^qh(Q_j(\alpha))+\sum_{i=1}^r h(P_i)+\log C\right)\notag\\
+\dfrac{(n+1)!}{\triangle N^{n+1}}(1+o(1))\sum_{v\in S}\log\tilde{h}_v(\alpha)
<\dfrac{\epsilon}{4}h(x(\alpha)),
\end{align*}
for all but finitely $\alpha\in B.$ Then for all but finitely $\alpha\in B$, we have
$$\sum_{v\in S}\sum_{j=1}^q\lambda_{Q_j(\alpha),v}(x(\alpha))
\leq d(m(n+1)+\epsilon)h(x(\alpha)).
$$

 \end{proof}

Now, we show how to deduce the main theorem from theorem \ref{t:31}.
\vskip0.2cm
\noindent
\textit{Proof of the Main Theorem.}

 We assume that
$$Q_j=\sum_{I\in\mathcal{T}_{d_j}}a_{j,I}x^I, a_{j,I}: \Lambda\rightarrow k, 1\leq j\leq q.$$  
Let $A\subset\Lambda$ be an infinite index subset that is  coherent with respect to $\{Q_j\}_{j=1}^q$. Since $x$ is algebraically non-degenerate over $V$ and $\mathcal{R}_{\{Q_j\}_{j=1}^q}$ then we have
$$A_j:=\{\alpha\in A|Q_j(x_0(\alpha),\ldots, x_M(\alpha))\not=0\}, 1\leq j \leq q$$
are infinite subsets of $A$.  Therefore, by passing to an infinite subset of $A$, we can assume that 
$$Q_j(x_0(\alpha),\ldots, x_M(\alpha))\not=0, \forall\alpha\in A, j=1,\ldots, q.$$

By coherence, for each $j\in\{1,\ldots,q\}$, there exists $I_j\in\mathcal{T}_{d_j}$ such that $a_{j,I_j}\not=0$ for all but finitely many $\alpha\in A.$ Set 
$$\tilde{Q}_j=\dfrac{Q_j}{a_{j,I_j}}=\sum_{I\in \mathcal{T}_{d_j}}\tilde{a}_{j,I}x^I, 1\leq j\leq q.$$
  
Let $d=lcm \{d_j\}_{j=1}^q.$ 
We have
$$\tilde{Q}_j^{d/d_j}=(\sum_{I\in\mathcal {T}_{d_j}}\tilde{a}_{j,I}x^I)^{d/d_j}=\sum_{\tilde{I}\in\mathcal {T}_d}g_{j,\tilde{I}}x^{\tilde{I}}.$$
By coherence of $A$ with respect to $\{Q_j\}_{j=1}^q$, A is also coherent with respect to $\{\tilde{Q}_j^{d/d_j}\}_{j=1}^q$. Then, for each $j\in\{1,\ldots,q\}$, there exists $\tilde{I}_j\in\mathcal{T}_d$ such that $g_{j,\tilde{I}_j}\not=0$ for all but finitely many $\alpha\in A.$  Consider  the set of polynomials $$\{\tilde{Q}^{d/d_j}_j/g_{j,\tilde{I}_j}\}_{j=1}^q.$$

  For all infinite subsets $B\subset A,$ we have
  $$\mathcal{R}_{B,\{Q_j\}_{j=1}^q}\supset\mathcal{R}_{B,\{\tilde{Q}_j\}_{j=1}^q}\supset \mathcal{R}_{B,\{\tilde{Q}^{d/d_j}_j\}_{j=1}^q}\supset\mathcal{R}_{B, \{\tilde{Q}^{d/d_j}_j/g_{j,\tilde{I}_j}\}_{j=1}^q}.$$
  Together with the fact that $x$ is algebraically non-degenerate over $V$ and $\mathcal{R}_{\{Q_j\}_{j=1}^q},$ we have  $x$ is algebraically non-degenerate over $V$ and $\mathcal{R}_{\{\tilde{Q}^{d/d_j}_j/g_{j,\tilde{I}_j}\}_{j=1}^q}.$ 
   Moreover, for each $1\leq j\leq q$, $\{\tilde{Q}^{d/d_j}_j/g_{j,\tilde{I}_j}\}_{j=1}^q$ has at least one coefficient equal to 1. 

   Therefore, we can apply theorem \ref{t:31} for $\{\tilde{Q}^{d/d_j}_j/g_{j,\tilde{I}_j}\}_{j=1}^q$ and collection of points $x: A\longrightarrow V(k)$. Then, we know that there exists an infinite index subset $B\subset A$ such that for all $\alpha\in B,$
  
  $$\sum_{v\in S}\sum_{j=1}^q\dfrac{1}{d}\log\dfrac{\|x(\alpha)\|_v^d\|\tilde{Q}^{d/d_j}_j(\alpha)\|_v}{\|\tilde{Q}^{d/d_j}_j(x(\alpha))\|_v}\leq (m(n+1)+\epsilon/3)h(x(\alpha)).$$
  Since for each $1\leq j\leq q,$ $\tilde{Q}_j $ has at least one coefficient equal to 1, we have
  $$h(\tilde{a}_{j,I}(\alpha))\leq h(\tilde{Q}_j(x(\alpha)))=o(h(x(\alpha))), I\in\mathcal{T}_{d_j}, \tilde{a}_{j,I}(\alpha)\not=0.$$
  Therefore, for all $ v\in M_k, $ we have
  $$\log\|\tilde{a}_{j,I}(\alpha)\|_v\geq - h(\tilde{a}_{j,I}(\alpha))\geq -o(h(x(\alpha))), I\in\mathcal{T}_{d_j}, \tilde{a}_{j,I}(\alpha)\not=0.$$
  Hence, for all $v\in M_k,$ we have
   $$\log\|\tilde{Q}^{d/d_j}_j(\alpha)\|_v\geq \dfrac{d}{d_j}\log\|\tilde{a}_{j,\hat{I}_j}(\alpha)\|_v\geq -o(h(x(\alpha))),$$ where $\hat{I}_j:=\max\{I| \tilde{a}_{j,I}(\alpha)\not=0\}$ (in the lexicographic order).
   Therefore, outside a finite subset of $B$, we have
   $$\sum_{v\in S}\sum_{j=1}^q\dfrac{1}{d}\log\dfrac{\|x(\alpha)\|_v^d}{\|\tilde{Q}^{d/d_j}_j(x(\alpha))\|_v}\leq (m(n+1)+2\epsilon/3)h(x(\alpha)).$$
  Since for each  $ 1\leq j\leq q, \tilde{Q}_j$ has at least one coefficient equal to 1, we have
  $$\sum_{v\in S} \log\|\tilde{Q}_j(\alpha)\|_v\leq h(\tilde{Q}_j(\alpha))=o(h(x(\alpha))), 1\leq j\leq q.$$
    Therefore, outside a finite subset of $B,$ we have
    $$\sum_{v\in S}\sum_{j=1}^q\dfrac{1}{d_j}\log\dfrac{\|x(\alpha)\|_v^{d_j}\|\tilde{Q}_j(\alpha)\|_v}{\|\tilde{Q}_j(x(\alpha))\|_v}\leq (m(n+1)+\epsilon)h(x(\alpha)),$$
   which implies the desired result.


\end{document}